\documentclass[a4paper,draft,oneside,12pt]{article}
\usepackage{ amsthm,amsmath,amsfonts,amssymb}
\usepackage{epsfig}
\usepackage{colordvi,helvet,multicol}
\usepackage{epic}
\usepackage{times}
\usepackage{graphics}
\usepackage{makeidx}
\usepackage{fancyhdr}
\makeindex
\long\def\@caption#1[#2]#3{
  \par
  \addcontentsline{\csname ext@#1\endcsname}{#1}%
    {\protect\numberline{\csname the#1\endcsname}{\ignorespaces #2}}%
  \begingroup
    \@parboxrestore
    \if@minipage
      \@setminipage
    \fi
   \normalsize
    \footnotesize
    \@makecaption{\csname fnum@#1\endcsname}{\ignorespaces #3}\par
  \endgroup}
\makeatother
 \theoremstyle{plain}

\newcommand{\ee}{\end{equation}}

\newtheorem{definition}{Definition}[section]
\newtheorem{theorem}[definition]{Theorem}
\newtheorem{lemma}[definition]{Lemma}
\newtheorem{remark}[definition]{Remark}
\newtheorem{proposition}[definition]{Proposition}
\newtheorem{corollary}[definition]{Corollary}

\usepackage{verbatim}
\newtheorem{pr1}[definition]{}
\usepackage{graphicx}

\usepackage{amsmath,amsfonts,amssymb}
\usepackage{latexsym}
\newcommand\RR{{\Bbb R}}
\newcommand\CC{{\Bbb C}}

\newcommand\NN{{\Bbb N}}

\begin{document}

\title{ On  the  $\alpha$-Amenability of  Hypergroups  
}

\author{Ahmadreza Azimifard \\
\footnotesize\texttt{} }

\date{}
 \maketitle

\maketitle

\begin{abstract} 
Let $UC(K)$ denote  the Banach space of all  bounded uniformly
 continuous  functions on a hypergroup  $K$. The main results of this article  
 concern on the $\alpha$-amenability of $UC(K)$ and quotients and products of hypergroups. 
  It is also  shown that 
 a Sturm-Liouville  hypergroup with a positive index  is $\alpha$-amenable if and only if $\alpha=1$.

\vspace{.5cm}

\footnotesize{
\begin{tabular}{lrl}
{\bf  Keywords.}  &  \multicolumn{2}{l} {\em Hypergroups: Sturm-Liouville, {Ch\'{e}bli-Trim\`{e}che}, 
Bessel-Kingman.}\\
 &  \hspace{-.1cm}  {\em $\alpha$-Amenable  Hypergroups.}\\
 \end{tabular}
 \vspace{.3cm}

 
}
\end{abstract}

\section{Introduction}

In \cite{Ska92} M. Skantharajah systematically studied the amenability of hypergroups.
Among other things, he obtained various equivalent statements on   the amenability of hypergroups.
Let $K$ be  a locally compact hypergroup. Let 
$L^1(K)$  and $UC(K)$  denote the hypergroup algebra and the  Banach space of all  bounded uniformly
 continuous  functions on $K$, respectively. He showed that $K$ is amenable if and only if there exists an invariant mean on 
$UC(K)$. Observe that, contrary to the group case,  $UC(K)$ may fail to be an algebra in general. 
Commutative or compact hypergroups are amenable, and the amenability of $L^1(K)$ implies the amenability of $K$;
 however, the converse is not valid any longer  even if $K$ is commutative (see also \cite{thesis, f.l.s}).

Recently the notion of $\alpha$-amenable hypergroups was introduced and studied in \cite{f.l.s}.
Suppose  that $K$ is commutative, let $\alpha\in \widehat{K}$, and  denote by  $I(\alpha)$   the maximal ideal in $L^1(K)$
generated by $\alpha$. As shown in \cite{f.l.s}, $K$ is $\alpha$-amenable 
if and only if either $I(\alpha)$ has a bounded approximate identity or $K$ satisfies  
the modified Reiter condition of $P_1$-type in $\alpha$. 
The latter condition together 
with the recursion formulas for orthogonal polynomials yields a 
sufficient condition for the $\alpha$-amenability of  a polynomial hypergroup.
However, this   condition is not  available for well known  hypergroups on the non-negative real axis.

 The   purpose of this article is    to generalize  the notion of $\alpha$-amenability for  $K$ to   the Banach space  $UC(K)$. 
 It then  turns out  that the $\alpha$-amenability of $K$ is equivalent to the $\alpha$-amenability of $UC(K)$, and 
 a  $\alpha$-mean on $UC(K)$ is unique if and only if  $\alpha$ belongs to $L^1(K)\cap L^2(K)$. Furthermore, some results are  obtained
 on  the $\alpha$-amenability of quotients and products of  hypergroups. Given
a Sturm-Liouville hypergroup $K$ with a  positive index, it is also shown that  there exist  non-zero point 
derivations on $L^1(K)$. Therefore, $L^1(K)$ is not weakly amenable, $\{\alpha\}$ $(\alpha\not=1)$ is not a spectral set,
 and   $K$ is not $\alpha$-amenable if $\alpha\not=1$. However, an example (consisting of a certain Bessel-Kingman hypergroup) shows that 
 in general $K$ is not necessarily $\alpha$-amenable if $\{\alpha\}$ is a spectral set.

This article is organized as follows: Section 2 collects pertinent concepts concerning on hypergroups.  Section 3 considers the $\alpha$-amenability of $UC(K)$. Section 4 contains  the $\alpha$-amenability of quotients and  products of  hypergroups, and  Section 5 is considered on the question of  $\alpha$-amenability of Sturm-Liouville hypergroups.

\section{Preliminaries}

            Let $(K, \omega, \sim )$ be a locally compact hypergroup, where 
            $\omega:K\times K\rightarrow M^1(K)$ defined by $(x,y)\mapsto \omega(x,y)$,
            and $\sim:K\rightarrow K$ defined by $x\mapsto \tilde{x}$,  
            denote the convolution and involution on $K$, where  $M^1(K)$ stands 
            for all probability measures on $K$.
             $K$ is called commutative if  $\omega{(x,y)}=\omega{(y,x)}$, for 
             every $x, y\in K$.

    Throughout the article $K$ is a commutative hypergroup. 
    Let $C_c(K)$, $C_0(K)$, and $C^b(K)$ be the
    spaces of all  continuous functions, those which have   compact support, 
    vanishing at infinity, and bounded on $K$ respectively; both $C^b(K)$ and $C_0(K)$  
    will be  topologized by the uniform norm
      $\left\| \cdot  \right\|_\infty$.  The  space of complex regular  Radon
      measures on $K$ will be denoted by $M(K)$,   which  coincides with the dual space of $C_0(K)$
      \cite[Riesz's Theorem (20.45)]{stromberg}.
      The translation of $f\in C_c(K)$ at  the point $x\in K$, $T_xf$,
      is defined by $T_xf(y)=\int_K  f(t)d\omega{(x,y)}(t)$, for every $y\in K$.
      Being $K$ commutative  ensures the existence of a Haar measure on $K$ which is unique up to 
      a multiplicative constant \cite{Spec75}.
      Thus, according to the translation
      $T$, let $m$ be the
      Haar measure on $K$, and let $(L^1(K), \left\|\cdot\right\|_1)$ denote the usual Banach space of all integrable functions on $K$ \cite[6.2]{Jew75}.
            For   $f, g\in L^1(K)$   the convolution and
      involution  are  defined  by
      $ f*g(x):=\int_K f(y)T_{\tilde{y}}g(x)dm(y)$ ($m$-a.e. on $K$) 
      and
      $f^\ast(x)=\overline{f(\tilde
      x)}$,  respectively, that $(L^1(K),\left \|\cdot \right\|_1)$ becomes a commutative Banach $\ast$-algebra.
       If $K$
      is discrete, then $L^1(K)$ has an identity;
      otherwise $L^1(K)$ has a b.a.i. (bounded approximate identity), i.e.
      there exists a net $\{e_i\}_i$ of functions in $L^1(K)$ with 
      $\|e_i\|_1\leq M$, for some $M>0$,
      such that $\|f \ast e_i-f\|_1\rightarrow 0$ as
      $i\rightarrow \infty$ \cite{BloHey94}.
      The dual space $L^1(K)^\ast$ can be identified with the space  $L^\infty(K)$  of 
      essentially bounded  Borel measurable complex-valued functions on $K$. 
                       The  bounded multiplicative linear
                      functionals on $L^1(K)$ can be identified
                      with
                      $$
                                     \mathfrak{X}^b(K):=\left\{
                                     \alpha\in C^b(K): \alpha\not=0, \;\omega(x,y)
                                    (\alpha)=\alpha(x)\alpha(y), \; \forall\;x,y\in K\right\},$$
    where $\mathfrak{X}^b(K)$ is a locally compact Hausdorff space with  the   compact-open topology.
     $\mathfrak{X}^b(K)$ with its subset
                   $$ \widehat{K}:=\{\alpha\in \mathfrak{X}^b(K):
                    \alpha(\tilde x)=\overline{\alpha(x)}, \;  \forall x\in K\}$$
    are considered as the  character spaces  of $K$ .

    The Fourier-Stieltjes transform of
   $\mu\in M(K)$, $\widehat{\mu}\in C^b(\widehat{K})$,  is  $\widehat{\mu}(\alpha):
   =\int_K \overline{\alpha(x)}d\mu(x)$,   which by restriction on  $L^1(K)$ it  is called
   Fourier transform and $\widehat{f}\in C_0(\widehat{K})$, for every  $f\in L^1(K)$.
   
    There exists a unique  regular positive Borel measure $\pi$ on $\widehat{K}$ 
 with the support $\mathcal{S}$ such that 
 \begin{equation}\label{ch.1.23}\notag
                          \int_K  \left| f(x)\right|^2 dm(x)=
                          \int_{\mathcal{S}}\left| \widehat{f}(\alpha)\right|^2 d\pi(\alpha),
               \end{equation}
                       for all $f\in L^1(K)\cap L^2(K)$. $\pi$ is called  Plancherel measure  and 
                       its support,
                        $\mathcal{S}$, is   a  nonvoid closed subset of $\widehat{K}$. Observe that  
                        the constant function $1$  is in general  not contained in 
                        $\mathcal{S}$. We have  
  $\mathcal{S}\subseteq \widehat{K}\subseteq \mathfrak{X}^b(K)$, where proper
 inclusions are possible;   see \cite[9.5]{Jew75}.

\section{ $\alpha$-Amenability of $UC(K)$}

\begin{definition}\label{ch.2.14}
\emph{Let $K$ be a  commutative hypergroup
                           and
                           $\alpha \in \widehat{K}$. Let $X$ be a subspace of
                           $L^\infty(K)$ with
                           $\alpha\in X$ which  is closed under complex conjugation
                           and is translation invariant.
                            Then $X$ is called $\alpha$-amenable if
                             there
                           exists a $m_\alpha\in X^\ast$ with
                           the following properties:}
\begin{itemize}
              \item[\emph{(i)}]$m_\alpha(\alpha)=1$,
                  \item[\emph{(ii)}]$m_\alpha(T_x f)
                        =
                          {\alpha(x)} m_\alpha (f), $ \hspace{.2in}\emph{for
                          every }$f\in X$ \emph{and }$x \in
                          K$.
                          \end{itemize}
 \end{definition}
        The hypergroup     $K$ is called $\alpha$-amenable
             if $X=L^\infty(K)$ 
             is $\alpha$-amenable;
             in the case   $\alpha=1$,  $K$ respectively $L^\infty(K)$ 
             is called amenable.  As shown in  \cite{f.l.s}, 
              $K$ is $\alpha$-amenable
              if and only if either $I(\alpha)$ has
              a b.a.i. or  $K$ has
             the modified Reiter's condition of $P_1$ type in the character  $\alpha$.
             The latter is also equivalent to the
             $\alpha$-left amenability of $L^1(K)$, if $\alpha$ is real-valued  \cite{thesis}.
             For instance, commutative hypergroups are amenable, and compact hypergroups are $\alpha$-amenable 
             for  every character $\alpha$.


If $K$ is a locally compact group, then the amenability of $K$ is
equivalent to the amenability of diverse subalgebras of $L^\infty(K)$,
 e.g. $UC(K)$ the
 algebra of  bounded uniformly  continuous functions  on $K$ \cite{Pat88}.  The same is 
 true for hypergroups  although   $UC(K)$ fails to
 be an algebra in general  \cite{Ska92}.  We now  prove  this fact  in   terms
  of  $\alpha$-amenability.

Let $UC(K):=\{f\in C^b(K): x \mapsto T_xf \mbox{ is continuous from  } K
\text{ to } (C^b(K), \left\|\cdot \right\|_\infty)\}$.
    The function space  $UC(K)$ is a norm closed, conjugate closed,
 translation invariant subspace of
$C^b(K)$ containing the constants and the continuous functions vanishing at 
infinity \cite[Lemma 2.2]{Ska92}.
 Moreover,
 $\mathfrak{X}^b(K)\subset UC(K)$ and
$UC(K)=L^1(K)\ast L^\infty(K)$. Let $B$ be a subspace of $L^\infty(K)$ such that
$UC(K)\subseteq B$.      $K$ is amenable if and only if $B$ is amenable \cite[Theorem 3.2]{Ska92}.
%
%
%
%
 
The following theorem  provides a further  equivalent statement to the $\alpha$-amenability of $K$.

\begin{theorem}\label{main.1}
\emph{Let $\alpha\in \widehat{K}$. Then $UC(K)$ is $\alpha$-amenable if and only if $K$ is
$\alpha$-amenable.}
\end{theorem}

\begin{proof}[ \emph{Proof:}]
Let $UC(K)$ be $\alpha$-amenable. There exists a $m_\alpha^{uc} \in UC(K)^\ast$ such
that $m_\alpha^{uc}(\alpha)=1$ and $m_\alpha^{uc}(T_xf)=\alpha(x)m_\alpha^{uc}(f)$ for
all  $f\in UC(K)$ and  $x\in K$. Let $g\in L^1(K)$ such that $\widehat{g}(\alpha)=1$.
 Define $m_\alpha:L^\infty(K)\longrightarrow \CC$ by
 \[m_\alpha(\varphi)=m_\alpha^{uc}(\varphi\ast g) \hspace{1.cm}(\varphi\in L^\infty(K))\]
 that $m_\alpha|_{UC(K)}=m_\alpha^{uc}$.
Since $\varphi\ast g\in UC(K)$,  $m_\alpha$ is a well-defined bounded linear
 functional on
$L^\infty(K)$, $m_\alpha(\alpha)=1$, and for  all $x\in K$ we have
\begin{align}
m_\alpha(T_x\varphi)&=m_\alpha^{uc}((T_x\varphi)\ast g)\\ \notag
&=m_\alpha^{uc}((\delta_{\tilde{x}}\ast \varphi)\ast g))\\ \notag
&=m_\alpha^{uc}(\delta_{\tilde{x}}\ast (\varphi\ast g))\\ \notag
&=m_\alpha^{uc}(T_x(\varphi\ast g))\\ \notag
&=\alpha(x)m_\alpha^{uc}(\varphi\ast g)\\ \notag
&=\alpha(x)m_\alpha(\varphi).
\end{align}
The latter shows that every $\alpha$-mean on $UC(K)$  extends  on $L^\infty(K)$.
 Plainly  the restriction of any  $\alpha$-mean of $K$   on   $UC(K)$ is  a $\alpha$-mean
 on $UC(K)$. Therefore, the statement is valid.
\end{proof}
%
%
\begin{corollary}
\emph{Let $\alpha\in \widehat{K}$ and $UC(K)\subseteq B\subseteq L^\infty(K)$. 
Then $K$ is $\alpha$-amenable if and only if $B$ is $\alpha$-amenable.}
\end{corollary}

The Banach space   $L^\infty(K)^{\ast}$ with the Arens product defined as follows is a Banach algebra:
\[\langle m\cdot m',  f\rangle=\langle m, m'\cdot f\rangle,  \mbox{ in which }
\langle m'\cdot f,  g \rangle=\langle m',  f\cdot g \rangle,\]
 and $\langle f\cdot g,  h\rangle=\langle f,  g \ast h\rangle$
 for all $m, m'\in L^\infty(K)^\ast$, $f\in L^\infty(K)$ and $g,h\in L^1(K)$  
 where $\langle f,g\rangle:=f(g)$.

The Banach space  $UC(K)^\ast$ with the restriction of the Arens product 
is a Banach algebra,  and  it can be identified with a
  closed right ideal of the Banach algebra
$L^1(K)^{\ast\ast}$ \cite{ulger}.

 If $m, m'\in UC(K)^\ast$, $f\in UC(K)$, and $x\in K$,  then $m'\cdot f\in UC(K)$ and we may have 
 \[\langle m\cdot m', f\rangle=\langle m,m'\cdot f\rangle, \hspace{.1cm}
                               \langle m'\cdot f, x\rangle =\langle  m',T_xf\rangle.\]
        If  $y\in K$,  then  $\int_K T_tf
                               d\omega(x,y)(t)=T_y(T_xf)$ which implies that 
                               \begin{equation}\notag
                               T_x(m\cdot f)=m\cdot T_xf,
                               \end{equation}
                                as
                               \begin{align}\notag
                               T_x(m\cdot f)(y)&=\int_K
                               m\cdot f(t)d\omega(x,y)(t)\\ \notag
                               &=\int_K \langle m, T_tf\rangle d\omega(x,y)(t)
                               \\ \notag
                               &=\langle m, \int_K T_tf
                               d\omega(x,y)(t)\rangle.  \notag
                               \end{align}


 Let $  X = UC(K)$,  $f\in X$  and  $g\in C_c(K)$ $(g\geq 0)$    
 with  $\|g\|_1=1$. Since the mapping $x\rightarrow T_xf$ is continuous 
 from $K$ to $(C^b(K), \left\|\cdot \right\|_\infty)$
 and the point evaluation functionals in $X^\ast$ separate points of $X$, we have
 \[g\ast f=\int_K g(x) T_{\tilde{x}}fdm(x).\]
 Therefore,  for every  1-mean $m$ on $X$  we have  
 $m(f)=m(g\ast f)$. 
 Hence,  two  1-means $m$ and $m'$  on 
 $L^\infty(K)$ are equal if  they are equal 
 on $UC(K)$, as
 \[m(f)=m(g\ast f)=m'(g\ast f)=m'(f) \hspace{1.cm}(f\in L^\infty(K))\]
 and $g$ is assumed as above. The latter  
 together with \cite[Theorem 3.2]{Ska92} show 
 a bijection  between means on 
 $UC(K)$ and $L^\infty(K)$. So, if   $UC(K)$ is amenable 
 with  a unique mean, then its extension on
 $L^\infty(K)$ is also unique which implies  that $K$ is compact  \cite{Kamyabi.2},  therefore 
 the identity character is isolated in $\mathcal{S}$,  the support of the 
 Plancherel measure.
 We   have  the following theorem in general:

\begin{theorem}\label{main.2}
\emph{If $UC(K)$ is $\alpha$-amenable with the  unique $\alpha$-mean
$m^{uc}_\alpha$,  then $m^{uc}_\alpha\in L^1(K)\cap L^2(K)$,
$\{\alpha\}$ is isolated in $\mathcal{S}$
and $m^{uc}_\alpha=\frac{\pi(\alpha)}{\|\alpha\|_2^2}$, where
$\pi:L^1(K)\rightarrow L^1(K)^{\ast\ast}$ is the canonical
embedding. If $\alpha$ is positive, then $K$ is compact.
}

\end{theorem}
\begin{proof}[\emph{Proof:}]
Let $m_\alpha^{uc}$ be the  unique   $\alpha$-mean on
 $UC(K)$, $n\in UC(K)^\ast$
and $\{n_i\}_i$ be a net  converging to $n$ in the
$w^\ast$-topology. Let  $x\in K$ and $f\in UC(K)$. Then
\begin{align}
\langle m_\alpha^{uc}\cdot n_i, T_xf\rangle
 &=\langle m_\alpha^{uc}, n_i\cdot(T_xf)\rangle\\ \notag
 &=\langle m_\alpha^{uc},T_x(n_i\cdot f)\rangle\\ \notag
 &=\alpha(x)\langle m_\alpha^{uc}, n_i\cdot f\rangle\\ \notag
 &=\alpha(x)\langle m_\alpha^{uc}\cdot n_i,
f\rangle. \notag
\end{align}
 For $\lambda_i=\langle n_i, \alpha\rangle\not=0$, since the  associated 
  functional to the character $\alpha$
on $UC(K)^{\ast}$ is multiplicative \cite{young.1},
$m_\alpha^{uc}\cdot n_i/\lambda_i$ is  a  $\alpha$-mean on $UC(K)$ which
is equal to  $m_\alpha^{uc}$. Then the mapping $n\rightarrow
m_\alpha^{uc}\cdot n$ defined on $UC(K)^\ast$ is  $w^\ast$-$w^\ast$
continuous, hence    $m_\alpha^{uc}$ is in the topological centre of
$UC(K)^\ast$, i.e. $M(K)$; see \cite[Theorem 3.4.3]{Kamyabi.2}.
  Since
   $\widehat{m^{uc}_\alpha}(\beta)=\delta_\alpha(\beta)$
   and $\widehat{m^{uc}_\alpha}\in C^b(\widehat{K})$, $\{\alpha\}$
   is an open-closed subset of $\widehat{K}$.  The algebra    $L^1(K)$
   is a
    two-sided closed ideal in $M(K)$ and the Fourier transform is injective, so
        $m^{uc}_\alpha$ and $\alpha$ belong to $L^1(K)\cap L^2(K)$.
    The  inversion
    theorem, \cite[Theorem 2.2.36]{BloHey94}, indicates
    $\alpha=\widehat{\alpha}^{\check{\;}}$, accordingly
     $\alpha\in \mathcal{S}$.
  Let  $m_\alpha:=\pi(\alpha)/\|\alpha\|_2^2$.   Obviously  $m_\alpha$ is a  $\alpha$-mean
  on $L^\infty(K)$,  and the  restriction
  of $m_\alpha$ on  $UC(K)$ yields  the desired unique
  $\alpha$-mean. 
  
  If $\alpha$ is positive, then
  $$\alpha(x)\int_K\alpha(t)dm(t)=\int_K
  T_x\alpha(t)dm(t)=\int_K\alpha(t)dm(t)$$
 which  implies that  $\alpha=1$, hence $K$ is compact.
 \end{proof}

Observe that in contrast to the case of  locally compact groups,  there exist noncompact hypergroups with  
unique $\alpha$-means. For example, for  little  q-Legendre polynomial hypergroups,  we have 
 $\widehat{K}\setminus{\{1\}}\subset L^1(K)\cap L^2(K)$; see  \cite{Frankcompact}. Therefore, 
by Theorem \ref{main.2},  $UC(K)$ and  $K$ are    $\alpha$-amenable with the  unique $\alpha$-mean $m_\alpha$, 
whereas $K$ has infinitely many 1-means \cite{Ska92}.


Let $\Sigma_\alpha(X)$ be the set of all $\alpha$-means on $X= L^\infty(K)$ or $UC(K)$. 
If $\alpha=1$,  then  $\Sigma_1(X)$ is  nonempty (as $K$ is commutative)  
$\mbox{weak}^{\ast}$-compact convex set in $X^\ast$  \cite{Ska92}.
If  $\alpha\not=1$ and $X$ is $\alpha$-amenable, then the same is true for  
$\Sigma_\alpha(X)$. 
 

\begin{theorem}
\emph{Let $X$ be  $\alpha$-amenable ($\alpha\not=1$). Then
$\Sigma_\alpha(X)$  is a
nonempty $\mbox{weak}^{\ast}$-compact convex subset of $X^\ast.$
Moreover,  $ \Sigma_\alpha(X)\cdot M / \langle M, \alpha\rangle\subseteq \Sigma_\alpha(X)$,  
for all  $M \in X^{\ast}$  with
 $\langle M,\alpha\rangle \not=0$. Furthermore, if $m_\alpha \in \Sigma_\alpha(X)$, 
 then  $m_\alpha^n=m_\alpha$  for all $n\in\NN$.}
\end{theorem}

\begin{proof}[\emph{Proof:}]
Let  $0\leq \lambda \leq 1$ and $m_\alpha$, $m'_\alpha\in \Sigma_\alpha(X)$. If
 $m''_\alpha:=\lambda m_\alpha+(1-\lambda)m'_\alpha$, then $m''_\alpha(\alpha)=1$
 and $m''_\alpha(T_xf)=\alpha(x)m''_\alpha(f)$,  for every $f\in X$  and $x\in K$;
 hence
$m''_\alpha\in  \Sigma_\alpha(X)$.

If  $\{m_i\}\subset \Sigma_\alpha(X)$ such that $m_i\overset{w^\ast}{\longrightarrow}m$, then
 $m\in\Sigma_\alpha(X)$.
 We have  $m(\alpha)=1$
and
\[m(T_xf)=\underset{i\rightarrow \infty}{\lim}m_i(T_xf)=\alpha(x)
\underset{i\rightarrow \infty}{\lim}m_i(f)=\alpha(x)m(f),\]
 for all $f\in X$ and $x\in K$.  Moreover,
  $\Sigma_\alpha(X)$ is
 $w^\ast$-compact by   Alaoglu's theorem \cite[p.424]{Dunford}.
Let $M\in X^{\ast}$ with $\lambda=\langle M,\alpha\rangle\not=0$. Then  $M':=m_\alpha\cdot M/\lambda$
  is a $\alpha$-mean on  $X$, as
 \[\langle m_\alpha \cdot M, T_xf\rangle=\langle m_\alpha, M\cdot T_xf\rangle=\langle m_\alpha, T_x(M\cdot f)\rangle=
 \alpha(x)\langle m_\alpha, M\cdot f\rangle=\alpha(x)\langle m_\alpha\cdot M, f\rangle.\]
 
  Since $g\cdot m_\alpha=m_\alpha\cdot g=\widehat{g^\ast}(\alpha)m_\alpha$ for all $g\in L^1(K)$, the continuity of the Arens product in the first variable on $X$ together with Goldstein's theorem  yield    $m_\alpha^2=m_\alpha$;  hence, $m_\alpha^n=m_\alpha$ for all $n\in \NN.$
 \end{proof}

 
 \begin{remark}
 \emph{Observe that  if $K$ is $\alpha$ and $\beta$-amenable, then  $m_\alpha\cdot m_\beta=m_\beta\cdot m_\alpha$ if and only if $m_\alpha\cdot m_\beta=\delta_\alpha(\beta)m_\alpha$, as
 \begin{align}\notag
\alpha(x)\langle m_\alpha\cdot m_\beta, f\rangle&=\alpha(x)\langle m_\alpha, m_\beta\cdot f\rangle\\ \notag
&= \langle m_\alpha, T_x (m_\beta\cdot f)\rangle\\ \notag
&= \langle m_\alpha, m_\beta\cdot T_xf\rangle\\ \notag
 &= \langle m_\alpha\cdot m_\beta, T_xf\rangle\\ \notag
 &=\langle m_\beta\cdot m_\alpha, T_xf\rangle\\ \notag
 &=\langle m_\beta, m_\alpha\cdot T_xf\rangle\\ \notag
 &=\langle m_\beta, T_x(m_\alpha\cdot f)\rangle\\ \notag
 &=\beta(x)\langle m_\beta, m_\alpha\cdot f\rangle \\ \notag
&=\beta(x)\langle m_\beta\cdot m_\alpha, f\rangle\hspace{1.cm} (f\in X, x\in K).\notag
\end{align}
  }
 \end{remark}

\section{$\alpha$-Amenability of Quotient Hypergroups}\label{sec.quotient.}
A closed nonempty subset $H$ of $K$ is called a subhypergroup
 if $H\cdot H=H$ and $\tilde{H}=H$,
 where $\tilde{H}:=\{\tilde{x}: x\in H\}$.
   Let $H$ be a  subhypergroup of  $K$. Then    $K/H:=\{x H: x\in K\}$
is a locally compact space with respect to the quotient topology. If
$H$ is a subgroup or   a compact subhypergroup of $K$, then
\begin{equation}\label{quotient}\notag
\omega(x H, y H):
=\int_K \delta_{z H}d\omega(x,y)(z)\hspace{.5cm} (x,y\in K)
\end{equation}
defines a hypergroup structure on $K/H$,  which agrees
 with the double coset
hypergroup $K//H$; see  \cite{Jew75}. 
The properties and duals of  subhypergroups and  qoutient of commutative hypergroups have been 
intensively studied by M. Voit \cite{Voit92, voit.properties}. 

\begin{theorem}\label{quotient.1}
\emph{Let $H$ be a subgroup or a compact subhypergroup   of $K$.
  Suppose $p:K\rightarrow K/H$ is the canonical projection, and
   $\widehat{p}:\widehat{K/H}\longrightarrow \widehat{K}$ is
    defined by $\gamma\longmapsto \gamma o p $.
  Then  $K/H$ is $\gamma$-amenable
  if and only if  $K$ is $\gamma o p$-amenable.}
\end{theorem}

\begin{proof}[\emph{Proof:}]

Let $K/H$ be $\gamma$-amenable. Then there exists a
 $M_\gamma :C^b(K/H)\longrightarrow \CC$
 such that $M_\gamma (\gamma)=1$, and
  $M_\gamma (T_{xH}f)=\gamma(xH)M_\gamma(f)$.

Since $H$ is
  amenable \cite{Ska92}, let $m_1$ be a mean on $C^b(H)$. For  $f\in C^b(K)$,  define
 \[f^1:K\longrightarrow \CC \hspace{.2cm} \mbox{ by }\hspace{.2cm}f^1(x):=\langle m_1, T_xf|_H\rangle.\]
The function  $f^1$ is continuous, bounded, and  since $m_1$ is a mean for $H$,  we have    
 \begin{align}\notag
 T_hf^1(x)&=\int_Kf^1(t)d\omega(h,x)(t)\\ \notag
 &         =\int_K \langle m_1, T_tf|_H  \rangle d \omega(h,x)(t)\\ \notag
 &         =\langle m_1, \int_K T_tf|_Hd\omega(h,x)(t) \rangle\\ \notag
 &=\langle m_1, T_h\left[T_xf|_H\right]\rangle =\langle m_1, T_xf|_H\rangle=f^1(x),
 \end{align}
 for all $h\in H$. Then according to the assumptions on  $H$,  \cite[Lemma 1.5]{voit.properties} implies that  $f^1$ is  
 constant on the cosets
  of $H$ in $K$.
 We may write $f^1=Fof,$ $F\in C^b(K/H).$
  Define 
  \[m:C^b(K)\longrightarrow \CC \hspace{.2cm} \mbox{ by }\hspace{.2cm}  m(f)=\langle M_\gamma, F \rangle
  .\]
  We have
  \[_{xH}F(yH)=T_xf^1(y)=\int_K \langle m_1,T_uf|_H \rangle d\omega(x,y)(u)=\langle m_1,
   \;T_y\left(T_xf\right)|_H\rangle, \]

  as  $u\rightarrow T_uf|_H$ is continuous from $K$ to $(C^b(H), \left\|\cdot \right\|_\infty)$ and
   the point evaluation
  functionals in $C^b(H)^\ast$ separates the points
  of $C^b(H)$; hence  $T_{xH} Fop =\left(T_xf\right)^1.$ Therefore,
  \[m( T_xf) =\langle M_\gamma, T_{xH}F \rangle=\gamma(xH)\langle M_\gamma, F\rangle=\alpha(x)m(f).\]
  Moreover,  $\langle m, \alpha \rangle=\langle M_\gamma, \gamma \rangle=1$, where $\gamma o p=\alpha$.
  Then  $m(T_xf)=\alpha(x)m(f)$ for all $f\in C^b(K)$ and $x\in K.$

  To show  the converse,  let $m_\alpha$ be a $\alpha$-mean on $C^b(K)$,
and define
\[M: C^b(K/H) \longrightarrow \CC \hspace{.3cm} \mbox{ by}\hspace{.3cm} \langle M,
f\rangle =\langle m_\alpha, fo p\rangle.\]
Since
\[T_{xH}f(yH)=T_{xH} f o p(y)=\int_{K/H}f d\omega (xH, yH)=\int_K fo p d\omega(x,y)
=T_x fop (y),\]
so $M(T_{xH}f)=\langle m_\alpha, T_xf o p\rangle =
\alpha(x)\langle m_\alpha, f o p \rangle=\alpha(x) \langle M,
 f\rangle$. Since  $\widehat{p}$ is an isomorphism, \cite[Theorem 2.5]{Voit92},
  and  $\gamma o p =\alpha$, we have
\[\langle M, \gamma \rangle =\langle m_\alpha, \gamma o p\rangle =\langle m_\alpha, \alpha\rangle =1.\]
Therefore,  $M$ is the desired $\gamma$-mean on $C^b(K/H)$. 
 \end{proof}

\begin{pr1}{\bf{Example:}}
   \emph{Let  $H$ be compact and $H'$ be discrete  commutative hypergroups.
    Let $K:=H \vee  H'$
   denotes the  joint
 hypergroup that $H$ is a subhypergroup of
$K$ and $K/H\cong H'$ \cite{vrem.joint}. The hypergroups $K$, $H$ and $H'$ are amenable, and $H$ 
is $\beta$-amenable for
 every $\beta\in \widehat{H}$ \cite{f.l.s}.
By Theorem \ref{quotient.1}, $K$
 is $\alpha$-amenable if and only if  $H'$ is $\gamma$-amenable ($\alpha=\gamma o p$).}

\end{pr1}

  \begin{remark}
 \emph{Let  $G$ be  a $[FIA]_B$-group \cite{mosak}. Then the space $G_B$,  $B$-orbits in $G$,
 forms a
 hypergroup \cite[8.3]{Jew75}.
    If   $G$ is  an amenable $[FIA]_B$-group, then  the hypergroup $G_B$ is amenable
     \cite[Corollary 3.11]{Ska92}.
    However, $G_B$ may   not be
     $\alpha$-amenable for $\alpha\in \widehat{G_B}\setminus\{1\}$.
    For example,  let  $G:=\RR^n$  and  $B$ be the group of rotations which acts on  $G$. Then
     the hypergroup $K:=G_B$
    can be identified with  the Bessel-Kingman hypergroup $\RR_0:=[0,\infty)$ 
    of order $\nu=\frac{n-2}{2}$. 
    Theorem \ref{r.6} will show that  if $n\geq 2$ then $\RR_0$ is $\alpha$-amenable if and only if $\alpha=1$.
     Observe that   $L^1(\RR_0)$ is an amenable 
      Banach algebra for  $n=1$ \cite{Wol84}; hence every maximal 
      ideal of $L^1(\RR_0)$ has a b.a.i.  \cite{BonDun73}, consequently  $G_B$ is 
      $\alpha$-amenable for every  $\alpha\in \widehat{G_B}$ \cite{f.l.s}.
       }
\end{remark}
%
%
%
%
Let $K$ and $H$ be hypergroups with left Haar measures. Then it is straightforward to
show that $K\times  H$ is a hypergroup with
a left Haar measure. If $K$ and $H$ are commutative hypergroups, then $K\times H$ is a commutative hypergroup with a Haar measure.
 As in the case of locally compact groups  \cite{dales}, 
we have the following isomorphism 
                     \begin{equation}\label{definition}
                     \phi: L^\infty (K) \times L^\infty (H) \longrightarrow 
                         L^\infty(K\times H)\hspace{.1cm} \mbox{ by }
                         (f, g )\longrightarrow \phi_{(f,  \;g)},
                         \end{equation}
where $\phi_{(f, g )}(x,y)=f(x)g(y)$ for all $(x,y)\in K\times H$. 
Let $(x',y')\in K\times H$ and $(f,g)\in L^\infty(K)\times L^\infty(H)$. Then
\begin{align}\notag
T_{(x',y')}\phi_{(f, g)} (x,y)&=\int_{K\times H}\phi_{(f, g)}(t,t')d \omega(x',x)\times \omega(y',y)(t,t')\\ \notag
&=\int_K\int_H f(t)g(t') d \omega(x',x)(t)d\omega(y',y)(t')\\ \notag
&=T_{x'}f(x)T_{y'}g(y)\\ \notag
&=\phi_{(T_{x'}f, T_{y'}g)}(x,y).\notag
\end{align}

\begin{theorem}
\emph{Let $K$ and $H$ be commutative hypergroups. Then
\begin{itemize}
    \item[(i)]  the map $\phi$ defined  in  (\ref{definition}) 
                   is a  homeomorphism between $\widehat{K}\times \widehat{H}$  and   $\widehat{K \times H}$,
    where 
    the dual spaces bear  the compact-open topologies.
    \item[(ii)]  $K\times H$ is $\phi_{(\alpha, \beta)}$-amenable if and only if  $K$  and $H$ are $\alpha$ and $\beta$-amenable respectively.
    \end{itemize}
     }
\end{theorem}
\begin{proof}[\emph{Proof:}] 
 (i) It is the special case of \cite[Proposition 19]{BonDun73}. 
 
 (ii) As in the case of    locally compact groups \cite{dales},  we have  
 $L^1(K\times H)\cong L^1(K)\otimes_p L^1(H)$,
  where $\otimes_p$ denotes the projection tensor product of two
Banach algebras. If $K$ is $\alpha$-amenable  and $H$ is
$\beta$-amenable,
 then $I(\alpha)$ and $I(\beta)$,  the maximal ideals
of $L^1(K)$ and $L^1(H)$ respectively, have b.a.i. \cite{f.l.s}.
Since $L^1(K)$ and $L^1(H)$ have
 b.a.i., $L^1(K)\otimes_p I(\beta)+I(\alpha)\otimes_p L^1(H)$
, the  maximal ideal   in $L^1(K\times H)$ associated to the character 
$\phi_{(\alpha,\beta)}$,
 has  a b.a.i. \cite[ Proposition 2.9.21]{dales},
 that equivalently   $K\times H$ is $\phi_{(\alpha,\beta)}$-amenable.

To prove the converse, suppose $m_{(\alpha,\beta)}$ is a  $\phi_{(\alpha,\beta)}$-mean on $L^\infty(K\times H)$, define 
   \[m_\alpha:L^\infty(K)\longrightarrow \CC  \hspace{.2cm} \mbox{ by } 
 \langle m_\alpha, f\rangle: =\langle m_{(\alpha,\beta)}, \phi_{(f,\beta)} \rangle.\] We have 
 \begin{align}\notag
 \langle m_\alpha, T_xf\rangle&=\langle m_{(\alpha,\beta)}, \phi_{(T_xf,\beta)} \rangle \\ \notag
 &=\langle m_{(\alpha,\beta)}, T_{(x,e)}\phi_{(f,\beta)} \rangle \\ \notag
 &=\alpha(x)\beta(e)\langle  m_{(\alpha,\beta)}, \phi_{(f,\beta)}\rangle \\ \notag
 &=\alpha(x)\langle  m_{(\alpha,\beta)}, \phi_{(f,\beta)}\rangle\\ \notag
 &=\alpha(x)\langle m_\alpha, f\rangle,
 \end{align}
  for all  $f\in L^\infty(K)$ and $x\in K$. Since $\langle m_\alpha, \alpha\rangle
  =\langle m_{(\alpha,\beta)}, \phi_{(\alpha,\beta)} \rangle=1$,   $K$ is $\alpha$-amenable.
  Similarly  $m_\beta:L^\infty(H)\longrightarrow \CC$ defined by $m_\beta(g):=\langle m_{(\alpha,\beta)}, \phi_{(\alpha,g)} \rangle$ is 
  a $\beta$-mean on $L^\infty(H)$, hence $H$ is $\beta$-amenable.
 \end{proof}

\begin{remark}
\emph{The proof of the  previous theorem may also follow from
  Theorem\ref{main.1} with
a modifying   \cite[Proposition 3.8]{Ska92}.}
\end{remark}

\section{$\alpha$-Amenability of Sturm-Liouville  Hypergroups }
\label{trimech.hypergroup}
Suppose  $A:\RR_0\rightarrow \RR$  is continuous,  positive,  and continuously 
differentiable on $\RR_0\setminus\{0\}$. 
Moreover, assume  that 
\begin{equation}\label{sturm}
\frac{A'(x)}{A(x)}=\frac{\gamma_0(x)}{x}+\gamma_1(x),
\end{equation}
for all $x$ in a neighbourhood of $0$, with   $\gamma_0\geq 0$ such that  
\begin{itemize}
\item[SL1]one of the following additional conditions holds.
		{
 			\begin{itemize}
					\item[SL1.1] $\gamma_0>0$ and $\gamma_1\in C^\infty(\RR)$, $\gamma_1$ being an odd function, or 
					\item[SL1.2] $\gamma_0=0$ and $\gamma_1\in C^1(\RR_0)$.
			\end{itemize}
		}
\item[SL2] There exists $\eta\in C^1(\RR_0)$ such that $\eta(0)\geq 0,$ $\frac{A'}{A}-\beta$ is
	 nonnegative and decreasing on $\RR_0\setminus{\{0\}}$,
	  and $q:=\frac{1}{2}\eta'-\frac{1}{4}\eta^2+\frac{A'}{2A}\eta$ is
	  decreasing on $\RR_0\setminus{\{0\}}$.
\end{itemize}

The  function $A$ is called  {Ch\'{e}bli-Trim\`{e}che}  if $A$ is a Sturm-Liouville function 
 of type SL1.1 satisfying the additional assumptions that the quotient $\frac{A'}{A}\geq 0$ is decreasing and that $A$ is increasing with
 $\underset{x\rightarrow \infty}{\lim A(x)}=\infty.$ In this case  SL2  is fulfilled with $\eta:=0$.

Let $A$ be a Sturm-Liouville function satisfying (\ref{sturm}) and SL2. Then there exists always  
a unique commutative  hypergroup structure on $\RR_0$ such
 that $A(x)dx$ is the Haar measure. A hypergroup established by this way is called a Sturm-Liouville hypergroup and it  will be denoted by 
  $(\RR_0, A(x)dx)$.  If $A$ is  a Ch\'{e}bli-Trim\`{e}che function, then the hypergroup $(\RR_0, A(x)dx)$ is called 
   Ch\'{e}bli-Trim\`{e}che hypergroup.

 The  characters of  $(\RR_0, A(x)dx)$ can be considered as
 solution $\varphi_\lambda$ of the  differential equation
 \[\left(\frac{d^2}{dx^2}+\frac{A'(x)}{A(x)}\frac{d}{dx}\right)\varphi_\lambda=-(\lambda^2+\rho^2)\varphi_\lambda, \hspace{.2cm}
  \varphi_\lambda(0)=1,  
 \hspace{.2cm}\varphi_\lambda'(0)=0,\]
 where  $\rho:=\underset{x\rightarrow
\infty}{\lim}\;\frac{A'(x)}{2A(x)}$,  
 and  $\lambda\in \RR_\rho:=
 \RR_0 \cup i[0, \rho]$; see \cite[ Proposition 3.5.49]{BloHey94}.    
 As shown in \cite[Sec.3.5]{BloHey94}, $\varphi_0$ is a   strictly positive character, and   
   $\varphi_\lambda$
 has the following integral representation  
 \begin{equation}\label{rep.1}
 \varphi_\lambda(x)=\varphi_0(x)\int_{-x}^xe^{-i\lambda t}d\mu_x(t)
  \end{equation} 
  where $\mu_x\in M^1([-x,x])$ for every  $x\in \RR_0$ and all  $\lambda\in \CC$ . In the 
   particular case $\lambda:=i\rho$,  the equality (\ref{rep.1}) yields
  $\left|\varphi_0(x)\right|\leq e^{-\rho x}$, as  $\varphi_{i\rho}=1$.




\begin{proposition}\label{estimation.1}
\emph{  Let $\varphi_\lambda$ be as above. Then 
 \[\left|\frac{d^n}{d\lambda^n}\varphi_\lambda(x)\right|\leq x^ne^{(|Im \lambda|-\rho)x}\]
for all $x\geq 0$,  $\lambda\in\CC$ and $n\in \NN$.}
\end{proposition}

 \begin{proof}[\emph{Proof:}]
    For all  $\lambda \in \CC\setminus{\{0\}}$ and $x>0$,  applying the Lebesgue   dominated convergence 
    theorem  \cite{stromberg}
    yields 
 \begin{equation}\notag
\frac{d^n}{d\lambda^n}\varphi_\lambda(x)=\varphi_0(x) \int_{-x}^{x} (-it)^n e^{-i\lambda t}d\mu_x(t),  
 \end{equation}
   
 hence 
 \begin{equation}\notag
 \left| \frac{d^n}{d\lambda^n}\varphi_\lambda(x) \right|\leq \varphi_0(x) x^n \int_{-x}^x 
 \left| e^{-i\lambda t} \right|d\mu_x(t)
 \end{equation}
 which implies  that $\left| \frac{d^n}{d\lambda^n}\varphi_\lambda(x) \right|\leq x^ne^{(|Im \lambda|-\rho)x}$.
 %
 %
 \end{proof}

To study the $\alpha$-amenability of  Sturm-Liouville  hypergroups,  we may require 
 the  following fact in general.
 The functional $D\in L^1(K)^\ast$ is
called a $\alpha$-derivation ($\alpha \in \widehat{K}$) on $L^1(K)$ if
\[D(f\ast g)=
\widehat{f}(\alpha)D_\alpha(g)+\widehat{g}(\alpha)D_\alpha(f)
\hspace{.5cm}  (f,g\in L^1(K)).\]
Observe that    if the maximal ideal 
 $I(\alpha)$
  has an approximate identity, then  $D|_{I(\alpha)}=0$.

\begin{lemma}\label{lemma.1}
\emph{Let $\alpha\in \widehat{K}$.  If $K$ is $\alpha$-amenable,  then every 
$\alpha$-derivation on $L^1(K)$ is zero.}
\end{lemma}
\begin{proof}[\emph{Proof:}]
 Let
$D\in L^1(K)^\ast$ be  a $\alpha$-derivation on $L^1(K)$. Since $I(\alpha)$ has
a b.a.i.\cite{f.l.s},  $D|_{I(\alpha)}=0$. Assume that
 $g\in L^1(K)$  with $\widehat{g}(\alpha)=1$.  Consequently, $g\ast g-g\in I(\alpha)$
 which implies that $D(g)=0$.
\end{proof}


\begin{theorem}\label{sturm.liouvill}
\emph{Let $K$ be the Sturm-Liouville  hypergroup
with $\rho>0$ and $\varphi_\lambda\in \widehat{K}$. Then $K$ is $\varphi_\lambda$-amenable if and only if 
$\lambda=i\rho$.  
 }
\end{theorem}

\begin{proof}[\emph{ Proof:} ]
  By Proposition \ref{estimation.1},
 the mapping
\[D_{\lambda_0}:L^1(K)\longrightarrow \CC, \hspace{.5cm} D_{\lambda_0}(f)= \frac{d}{d\lambda}
\widehat{f}(\lambda)\Big|_{\lambda=\lambda_0} (\lambda_0\not=i\rho),\]
is a well-defined  bounded nonzero $\varphi_{\lambda_0}$-derivation.
 In that $\RR_0$ is  amenable \cite{Ska92}, applying Lemma \ref{lemma.1} will
  indicate that $K$ is $\varphi_\lambda$-amenable if and only if 
$\lambda=i\rho$.
\end{proof}
\begin{remark}

\begin{itemize}

\item[\emph{(i)}]\emph{ Let $K$ be 
the Sturm-Liouville  hypergroup
with $\rho>0$. By the previous theorem,   $L^1(K)$
is not weakly amenable  as well as  $\{\varphi_\lambda\}$, $\lambda\not=i\rho$,
is not a spectral set.\\
Theorem \ref{r.6} will show that  the maximal ideals
   associated to the spectral sets do not  have  b.a.i. necessarily.}
   
\item[\emph{(ii)}]\emph{
 A  Sturm-Liouville hypergroup is of
  exponential growth if and only if  $\rho>0$ \cite[Proposition 3.5.55]{BloHey94}. 
Then by Theorem \ref{sturm.liouvill}   such hypergroups are $\varphi_\lambda$-amenable if and only if  $\lambda=i\rho$. 
However,  for  $\rho=0$   we do not have a certain assertion. }

\end{itemize}
 \end{remark}
 We now study    special cases of Sturm-Liouville hypergroups in more details.

 \begin{itemize}
\item[(i)]{\bf{Bessel-Kingman hypergroup}\footnote{This subsection is
  from parts of the  author's Ph.D.
          thesis at the Technical University of Munich.}}

The Bessel-Kingman hypergroup
is a    {Ch\'{e}bli-Trim\`{e}che} hypergroup    on $\RR_0$ with
$A(x)=x^{2\nu+1}$ when  $\nu\geq -\frac{1}{2}$.
The characters    are  given  by
\[\alpha_{\lambda}^{\nu}(x):=2^\nu \Gamma
{(\nu+1)}J_{\nu}(\lambda x){(\lambda x)}^{-\nu},\] where $J_{\nu}(x)$ is the
Bessel
 function of order $\nu$,  and   $\lambda\in \RR_0$ represents  the characters. The dual 
 space $\RR_0$ has also a hypergroup
   structure and the bidual space coincides with  the hypergroup $\RR_0$
   \cite{BloHey94}.
As  shown in \cite{Wol84},  the $L^1$-algebra of  $(\RR_0, dx)$, the Bessel-Kingman hypergroup of
 order $-\frac{1}{2}$,
is amenable; as a result,  $(\RR_0, dx)$ is
$\alpha_{\lambda}^{\nu}$-amenable
 for every $\lambda\in \RR_0$.

Suppose     $L^1_{rad}(\RR^n)$ is   the  subspace of $L^1(\RR^n)$ of
radial functions  and $\nu=\frac{n-2}{2}$.  It is  a closed  self
adjoint subalgebra of $L^1(\RR^n)$ which  is   isometrically
$\ast$-isomorphic
 to the hypergroup algebra  $L^1(\RR_{0}, dm_n)$, where $dm_n(r)=
 \frac{2\pi^{d/2}}{\Gamma(n/2)}r^{n-1}dr$.

\begin{theorem}\label{r.6}
\emph{Let  $K$ be the Bessel-Kingman hypergroup of order $\nu\geq 0$.
If $\nu=0$ or $\nu\geq \frac{1}{2}$, then $K$ is $\alpha_{\lambda}^{\nu}$-amenable if 
and only if $\alpha_{\lambda}^{\nu}=1$.}
\end{theorem}

\begin{proof}[{\bf Proof:}]
 (i) Let $\nu=0$ and  $\alpha_{\lambda}^{0}\in \widehat{K}$. Since $K$ is commutative,
  $K$ is (1-)amenable \cite{Ska92}. Suppose now  $\alpha_{\lambda}^{0}\not=1$ and $K$  is
 $\alpha_{\lambda}^{0}$-amenable, so    $I(\alpha_{\lambda}^{0})$
  has a b.a.i.  If  $I_r(\alpha_{\lambda}^{0})$ is  the corresponding ideal to 
$I(\alpha_{\lambda}^{0})$ in $L^1_{rad}(\RR^2)$, then $I_r(\alpha_{\lambda}^{0})$
  has a b.a.i.,  say  $\{e'_i\}$.  Let $I:=[I_r(\alpha_{\lambda}^{0})\ast
L^1({\RR}^2)]^{cl}$.
   The group $\RR^2$ is amenable \cite{Pat88},  so let   $\{e_i\}$ be   a  b.a.i for
   $L^1(\RR^2)$.  For
every $f\in I_r(\alpha_{\lambda}^{0})$ and $g\in L^1({\RR}^2)$, we have
\begin{align}\notag
\left\| f\ast g-(f\ast g) \ast (e'_i\ast e_i) \right\|_1
 &
 \leq \left\| g\right\|_1
\left\|f\ast e'_i-f\right\|_1 + \left\|f\ast
e'_i\right\|_1\left\|g-g\ast e_i\right\|_1.\notag
\end{align}
The latter shows that $\left\{e'_i\ast e_i\right\}$ is  a b.a.i for
the closed ideal $I$. In
 \cite[Theorem 17.2]{Rei70}, it is shown that
      $Co(I)$, cospectrum of $I$,
is a finite union of lines  and points in $\RR^2$. But this
contradicts the fact that $Co(I)$ is a  circle with radius $\lambda$ in $\RR^2$; hence, 
 $K$ is $\alpha_{\lambda}^{0}$-amenable if and  only if $\alpha_{\lambda}^{0}=1$.

(ii) Following  \cite{schwartz.2},
$\widehat{f}(\lambda)=\int_0^\infty f(x) \alpha_\lambda^\nu(x)dm(x)
$
 is
differentiable,  for all  $f\in L^1(K)$,   $\nu\geq \frac{1}{2}$, and
$\lambda\not=0$. Since
$\frac{d}{dx}\left(x^{-\nu}J_\nu(x)\right)=-x^{\nu}J_{\nu+1}(x)$ and
$J_\nu(x)=\mathcal{O}(x^{-1/2})$ as  $x\rightarrow \infty$
\cite{AbrSte72},
there exists a constant  $A_\nu(\lambda_0)>0$ such that $ \left| \frac{d}{d\lambda}
\widehat{f}(\lambda)\Big|_{\lambda=\lambda_0}\right |\leq
A_\nu(\lambda_0)\left\|f\right\|_1\hspace{.1cm}$
$(\lambda_0\not=0)$. Hence, the mapping
\[D_{\lambda_0}:L^1(\RR_0^{\nu}, x^{2\nu+1}dx)\longrightarrow
\CC, \hspace{.4cm} D_{\lambda_0}(f)= \frac{d}{d\lambda}
\widehat{f}(\lambda)\Big|_{\lambda=\lambda_0},\] is a  well-defined
 bounded nonzero $\alpha_\lambda^\nu$-derivation. Hence,  Lemma
\ref{lemma.1} implies that   $K$  is $\alpha_\lambda^\nu$-amenable if and only if $\alpha_\lambda^\nu=1$.
\end{proof}

 \item [(ii)] {\bf{Jacobi hypergroup of noncompact type}}

 The Jacobi hypergroup
of noncompact type is a  {Ch\'{e}bli-Trim\`{e}che}   hypergroup
 with
\begin{equation}\label{J.non.3}\notag
 A^{(\alpha,\beta)}(x):=2^{2\rho}\sinh^{2\alpha+1}(x).\cosh^{2\beta+1}(x),
\end{equation}
where $\rho=\alpha+\beta+1$ and $\alpha\geq \beta \geq
-\frac{1}{2}$.    The  characters are  given by  Jacobi functions of
order $(\alpha, \beta)$, ${\varphi}^{(\alpha,\beta)}_\lambda(t):=
_2F_1(\frac{1}{2}(\rho+i\lambda), \frac{1}{2}(\rho-i\lambda);
\alpha+1; -\mbox{sh}^2t),$ where $_2F_1$ denotes the Gaussian
hypergeometric function, $\alpha\geq \beta \geq -\frac{1}{2}$, $t\in
\RR_0$ and $\lambda$ is  the parameter of character
${\varphi}^{(\alpha,\beta)}_\lambda$ which varies on $\RR_0\cup
[0,\rho]$. It is straightforward to show that
$\rho=\alpha+\beta+1$.
  As we have seen,  if $\rho>0$,  then $\RR_0$ is ${\varphi}^{(\alpha,\beta)}_\lambda$-amenable if and only  if ${\varphi}^{(\alpha,\beta)}_\lambda=1$.  If $\rho=0$, then $\alpha=\beta=-\frac{1}{2}$, hence
$\varphi_\lambda ^{(-\frac{1}{2},-\frac{1}{2})}(t)=\cos(\lambda
t)$ so that  turns $\RR_0$ to
 the Bessel-Kingman
hypergroup   of order $\nu=-\frac{1}{2}$, which is
 $\varphi_\lambda^{(\alpha,\beta)}$-amenable for every $\lambda$.

Let $A(\alpha,\beta):=\left\{\check{\varphi}:\hspace{2.mm} \varphi
\in L^1(\widehat{\RR_0} ,  \pi )\right\}$. By the inverse theorem,
\cite[Theorem 2.2.32]{BloHey94},
we have 
$A(\alpha,\beta)\subseteq C_0([0, \infty))$.  There exists a
convolution structure on $\RR_0$ such that $\pi$ is
 the  Haar measure
on $\RR_0$ and  $A(\alpha,\beta)$  is a Banach algebra of functions
on $[0, \infty)$;  see \cite{Flensted.inverse}.
The following estimation  with Lemma \ref{lemma.1}  will indicate
that for $\alpha\geq \frac{1}{2}$ and  $\alpha \geq \beta \geq
-\frac{1}{2}$, the maximal ideals of $A(\alpha,\beta)$ related to
the points in $(0, \infty)$ do not have  b.a.i.
\begin{theorem}\emph{\cite{Meaney}}
\emph{For $\alpha\geq \frac{1}{2}, \;\alpha \geq \beta \geq -\frac{1}{2}$ and $\varepsilon >0$
there is a constant $k>0$ such that  if
$f\in A(\alpha,\beta)$ then
$f|_{[\varepsilon, \infty)}\in C^{[\alpha+\frac{1}{2}]}([\varepsilon, \infty))$ and }
\begin{equation}\notag
\underset {t\geq \varepsilon}{\sup}\left|f^{(j)}(t)
\right|\leq k\left\|f\right\|_{(\alpha,\beta)}, \hspace{1.cm}0\leq j\leq [\alpha+\frac{1}{2}].
\end{equation}
\end{theorem}

\begin{remark} \emph{ Here are  the special cases of Jacobi hypergroups of noncompact type which are $1$-amenable
 only:}
\begin{itemize}
	\item[\emph{(i)}] \emph{Hyperbolic hypergroups, if   $\beta=-\frac{1}{2}$ and $\rho=\alpha+\frac{1}{2}>0$.} 
		\item[\emph{(ii)}] \emph{Naimark hypergroup, if    $\beta=-\frac{1}{2}$  and   $\alpha=\frac{1}{2}$.}
\end{itemize}
\end{remark}

\item[(iii)] {\bf{ Square hypergroup}}

The  Square hypergroup  is a Sturm-Liouville hypergroup on $\RR_0$
with $A(x)=(1+x)^2$ for all $x\in \RR_0$. Obviously  we have   $\rho=0$ and  the characters are given by 
$$\varphi _\lambda(x):=
\begin{cases}
\frac{1}{1+x}\left(\cos(\lambda x)+\frac{1}{\lambda}\sin(\lambda
x)\right)
&\text{if $\lambda\not=0$}\\
1      &\text{if $\lambda=0$}.
\end{cases}
$$

If $\lambda\not=0$,  then
$$ \frac{d}{d\lambda}\varphi _\lambda(x) =\frac{x}{1+x}\left(-\sin(\lambda x)+
\frac{\cos(\lambda x)}{\lambda}-\frac{1}{x{\lambda}^{2}}
\sin(\lambda x)\right),$$
 which is  bounded as  $x$ varies. Therefore, 
  the mapping
\[D_{\lambda_0}:L^1(\RR_0,Adx)\longrightarrow \CC,
 \hspace{1.cm}D_{\lambda_0}(f):=
 \frac{d}{d\lambda} \widehat{f}(\lambda)\Big|_{\lambda=\lambda_0},
\]
 is a
 well-defined  bounded nonzero  $\varphi_\lambda$-derivation on $L^1(\RR_0, Adx)$.
Applying  Lemma  \ref{lemma.1}  results that      ${\RR_0}$ is $1$-amenable only.

\end{itemize}

\begin{center}
Acknowldegment
\end{center}
The author would like to  thank the referee for valuable suggestions, which in particular led to a simplifying  the proof of   Proposition \ref{estimation.1}. 
%
%
%

\end{document}